\documentclass{amsart}

\usepackage{amsmath,amssymb,amsfonts}
\usepackage{ifthen}
\usepackage{graphicx}
\usepackage{float}
\usepackage[colorlinks=true]{hyperref}
\hypersetup{urlcolor=blue, citecolor=red}

\newtheorem{theorem}{Theorem}[section]
\newtheorem{lemma}[theorem]{Lemma}
\theoremstyle{definition}
\newtheorem{definition}[theorem]{Definition}
\newtheorem{remark}{Remark}

\title[Local Uniqueness of the Circular Integral Invariant]
      {Local Uniqueness of the Circular Integral Invariant}
      
\author[Martin Bauer, Thomas Fidler and Markus Grasmair]{}


\email{Bauer.Martin@univie.ac.at}
\email{Thomas.Fidler@univie.ac.at}
\email{Markus.Grasmair@univie.ac.at}

\thanks{The first author is supported by the Austrian Science Fund (FWF) project P 21030-N13.}

\DeclareMathOperator{\comp}{Comp}
\DeclareMathOperator{\diff}{Diff}
\DeclareMathOperator{\kernel}{Ker}

\DeclareMathOperator{\inn}{Int}
\DeclareMathOperator{\area}{area}
\DeclareMathOperator{\Emb}{Emb}
\DeclareMathOperator{\SE}{SE}
\DeclareMathOperator{\id}{Id}

\begin{document}
\maketitle

\centerline{\scshape Martin Bauer}
\medskip
{\footnotesize
 \centerline{Department of Mathematics, University of Vienna}
 \centerline{Nordbergstr.~15, A-1090 Wien, Austria}
} 

\medskip

\centerline{\scshape Thomas Fidler}
\medskip
{\footnotesize
 \centerline{Computational Science Center, University of Vienna}
 \centerline{Nordbergstr.~15, A-1090 Wien, Austria}
}

\medskip

\centerline{\scshape Markus Grasmair}
\medskip
{\footnotesize
 \centerline{Computational Science Center, University of Vienna}
 \centerline{Nordbergstr.~15, A-1090 Wien, Austria}
}

\bigskip

\begin{abstract}
  This article is concerned with the representation of curves by means of integral invariants. In contrast to
  the classical differential invariants they have the advantage of being less sensitive with respect to noise.
  The integral invariant most common in use is the circular integral invariant. A major drawback of this curve
  descriptor, however, is the absence of any uniqueness result for this representation. This article serves as
  a contribution towards closing this gap by showing that the circular integral invariant is
  injective in a neighbourhood of the circle. In addition, we provide a stability estimate valid on this neighbourhood.
  The proof is an application of Riesz--Schauder theory and the
  implicit function theorem in a Banach space setting.
\end{abstract}

\section{Introduction}

In many applications one faces the challenge to model objects, or parts of objects, in a mathematical framework.
As an example, one important task is to extract an object from a given data set and manipulate it in a
post-processing step in order to obtain further information. Typical applications include medical imaging, object
tracking in a sequence of images, but also object recognition, where the post-processing step consists of the
comparison of the extracted object with a database of reference objects. Similarly, such a comparison can be
necessary in medical imaging in order to distinguish between healthy and diseased organs. To that end, however,
one has to be able to decide whether two given objects are similar or not. This requires a representation of
the objects that makes the application of standard similarity measures possible.

Finding a suitable representation of the object of interest, depending on the type of application, is crucial
as a first step. For simplicity, it is often assumed that the object is a simply connected bounded domain,
allowing for the identification of the domain with its boundary. From a mathematical point of view, this
assumption reduces the complexity of the representation. In addition, there exists a larger number of
descriptions of boundaries than of domains, and, consequently, more mathematical tools to analyze the geometry
of the underlying objects.
\bigskip

In 2D a common approach is to encode the contour of an object by the curvature function of its boundary curve.
This approach has, for instance, been used in \cite{KlaSriMioJos04}, where the authors set up a shape space of
planar curves, where the shapes are implicitly encoded by the curvature function. The main advantage of using
a differential invariant --- the most prominent representative being the curvature --- to represent an object
is the well investigated mathematical framework of this type of invariants (see \cite{Car35,Lie84,Olv95}).
\smallskip

Since all kinds of differential invariants are based on derivatives, they suffer from the shortcoming of being
sensitive with respect to small perturbations. To bypass this shortcoming Manay et al.~\cite{ManHonYezSoa04}
proposed to use integral invariants instead of their differential counterparts
(see also \cite{FidGraSch08,HuaFloGelHofPot06,YanLaiHuPot06}). Integral invariants have similar invariance
properties as differential invariants, but have proven to be considerably more robust with respect to noise.
Their theory, however, is not that well investigated as opposed to the theory of their differential counterparts.
\smallskip

Beside the classical approach of differential invariants and the novel approach of integral invariants, there
exist several other concepts for encoding an object. For instance, in \cite{DucYezMitSoa03} the authors use
the zero level set of a harmonic function, which is uniquely determined by prescribing two functions on the
boundary of an annulus, to encode the boundary of a 2D object (see \cite{DucYezSoaRoc06} for a generalization
to compact surfaces in 3D). A similar encoding of the object by a function is given in the article of Sharon
and Mumford \cite{ShaMum06}. Here, the authors first map the 2D object, which is supposed to be a smooth and
simply closed curve, to the interior of the unit disc in the complex plane via the Riemann mapping theorem.
This conformal mapping is composed with a second one, generated out of the exterior of the original object,
and the composition is restricted to the boundary of the unit disc. Thus, the final mapping, which the authors
call the fingerprint of the object, is a diffeomorphism from the unit circle onto itself.
\bigskip

One of the challenges in object encoding is the question of uniqueness of the encoding. More precisely, in many
applications, e.g.~object matching, the correspondence between the object and its encoding should be one-to-one.
Thus, a thorough investigation of the operator that maps an object to its encoding is needed. In case of the
encoding by a harmonic or conformal mapping --- if possible --- uniqueness is well known. Also for the encoding
of an arc length parameterized curve by its curvature function, it is known that one obtains a one-to-one
correspondence between the curve and its encoding (up to rigid body motions). One even has a complete
characterization of the set of functions that arise as curvature functions of a class of sufficiently regular
curves (see \cite{Dah05}). For integral invariants the situation is different; the cone area invariant, first
introduced in \cite{FidGraSch08}, is an injective mapping independent of the space dimension, but its application
is limited to star-shaped objects. In contrast, for the circular integral invariant, which is the integral
invariant most common in use, there exists no proof for the uniqueness conjecture so far.
\smallskip

This article is a contribution towards this goal: We first prove that the integral invariant
is $\ell$-times continuously Fr\'echet differentiable in a neighborhood of the circle, seen as a mapping from
$C^{k+\ell+1}$ to $C^k$, $k \ge 0$. 
Then we show that the Fr\'echet differential is injective on some $C^{k+\ell+1}$-neighborhood of the circle,
$k \ge 1$, $\ell \ge 1$.
The proofs of these results are based on the implicit
function theorem on Banach spaces and an application of Riesz--Schauder theory.
Using the injectivity result, a Taylor series expansion, and an interpolation inequality
for $C^m$ norms, we obtain the injectivity of the integral invariant
on a $C^{k+6}$ neighborhood $\mathcal{V}$, $k \ge 1$.
More precisely, we show that, in $\mathcal{V}$,
the $C^k$-norm of the difference of the integral invariants of two curves
can be estimated from below by their $C^k$-distance.

\section{Setting}
Let $\Emb$ be the space of all continuous embeddings from $S^1$ to $\mathbb{R}^2$.
Then every curve $\gamma\in\Emb$ has a unique interior, denoted by $\inn(\gamma)$.
Following \cite{FidGraSch08,ManHonYezSoa04}, this allows us to introduce the circular integral invariant:

\begin{definition}
  For given $r>0$ we define the \emph{circular integral invariant}
  \begin{equation*}
    I_{r}[\gamma] \colon S^{1} \to \mathbb{R}
  \end{equation*}
  of a curve $\gamma\in\Emb$ as
  \begin{equation*}
    I_r[\gamma](\varphi) := \area\bigl(B_r\bigl(\gamma(\varphi)\bigr)\cap\inn(\gamma)\bigr)\,,
  \end{equation*}
  where $B_r(p)$ denotes the ball of radius $r$ centered at $p \in \mathbb{R}^{2}$.
\end{definition}

The circular integral invariant behaves well under several group actions:
\begin{itemize}
\item $I_r$ is invariant with respect to Euclidean motions: For $A\in \SE[2]$ we have
  \begin{equation*}
    I_r[A\circ \gamma] = I_r[\gamma] \;.
  \end{equation*}
\item $I_r$ is equivariant with respect to reparametrizations: For every homeomorphism
  $\Phi\colon S^1 \to S^1$ we have
  \begin{equation*}
    I_r[\gamma\circ\Phi] = I_r[\gamma]\circ\Phi \;.
  \end{equation*}
\item For every scalar $t>0$ we have
  \begin{equation*}
    I_r[t\gamma] = t^2 I_{r/t}[\gamma] \;.
  \end{equation*}
\end{itemize}

The observations above suggest to consider the integral invariant on the space $\mathcal{C}$ of all curves
modulo Euclidean motions and reparametrizations. Moreover, we assume as an additional smoothness property
that the considered curves are of class $C^{k}$, $k \geq 1$. Then it makes sense to use the following representation of
$\mathcal{C}$, as it avoids working with equivalence classes of curves.

\begin{definition}
  Denote by $\mathcal{C}_{k} \subset \Emb$, $k \geq 1$, the space of all curves $\gamma \in C^{k}(S^1;\mathbb{R}^2)$
  satisfying the following conditions:
  \begin{itemize}
  \item $\gamma$ has constant speed, i.e., there exists a constant $c_{\gamma} > 0$ such that
    $\lVert \dot\gamma(\varphi) \rVert = c_{\gamma}$ for all $\varphi \in S^1$.
  \item $\gamma(0) = (1,0)$ and $\dot\gamma(0) = (0,c_{\gamma})$, where we identify the circle
    $S^1$ with the interval $[0,2\pi)$.
  \item $\gamma$ is an embedding, that is, $\gamma(\varphi) \neq \gamma(\psi)$ for all $\varphi\neq\psi$.
  \end{itemize}
\end{definition}

For the proof of our main theorem we need the following result from differential geometry concerning the
manifold structure of $\mathcal{C}_{k}$:

\begin{theorem} \label{thm:smooth_submanifold}
  For $k \geq 1$ the space $\mathcal{C}_{k}$ is a smooth submanifold of the Banach space of all $C^k$-curves from $S^1$ to
  $\mathbb{R}^2$. Its tangent space $T_{\gamma} \mathcal{C}_{k}$ at a curve $\gamma\in\mathcal{C}_{k}$ consists
  of all $C^k$-curves $\sigma$ with
  \begin{equation*}
    \langle\dot\sigma(\varphi),\dot\gamma(\varphi)\rangle=c \text{ for some }
    c \in \mathbb{R},\quad \sigma(0) = (0,0) \text{ and } \langle \dot \sigma(0),\gamma(0) \rangle = 0 \;.
  \end{equation*}
\end{theorem}

\begin{proof}
  The proof of the submanifold result is similar to \cite[Thm.~2.2]{Pre11}. In our case the situation is less
  complicated, as we only deal with $C^{k}$-curves instead of Sobolev curves. The constant speed parameterization
  yields the condition
  \begin{equation*}
    2c
    =
    \partial_{\varepsilon}|_0
    \langle
    \dot{\gamma}(\varphi) + \varepsilon \dot{\sigma}(\varphi),\dot{\gamma}(\varphi) + \varepsilon \dot{\sigma}(\varphi)
    \rangle
    =
    2 \langle \dot\sigma(\varphi) ,\dot\gamma(\varphi)\rangle \;.
  \end{equation*}
  The remaining constraints follow directly from the initial conditions.
\end{proof}

Under additional smoothness assumptions on $\gamma$
we obtain the following characterization of the tangent space $T_{\gamma} \mathcal{C}_{k}$.

\begin{lemma} \label{lem:tangent_space_characterization}
  Let $k \ge 1$ and $\gamma \in \mathcal{C}_k \cap C^{k+1}(S^1;\mathbb{R}^2)$ with curvature function
  \begin{equation*}
    \kappa_{\gamma}(\varphi) := \frac{\langle \dot{\gamma}(\varphi)^{\bot},\ddot{\gamma}(\varphi) \rangle}{c_{\gamma}^{3}}.
  \end{equation*}
  Then the tangent space $T_\gamma\mathcal{C}_k$ consists of all $C^k$-curves 
  $\sigma(\varphi) = a(\varphi)\dot\gamma(\varphi)^{\bot}+b(\varphi)\dot\gamma(\varphi)$
  satisfying:
  \begin{itemize}
  \item $\dot{b}(\varphi) = \dot{b}(0) + a(\varphi)\kappa_{\gamma}(\varphi)c_{\gamma}$.
  \item $a(0) = b(0) = 0$.
  \item $\dot a(0) = 0$.
  \end{itemize}
\end{lemma}

\begin{proof}
  Theorem \ref{thm:smooth_submanifold} and the fact that
  $\langle \dot{\gamma}(\varphi),\ddot{\gamma}(\varphi) \rangle = 0$ imply that there exists a constant $c \in \mathbb{R}$
  such that
  \begin{align*}
    c	&= \langle \dot{\sigma}(\varphi) ,\dot{\gamma}(\varphi)\rangle
    = \langle \dot{\gamma}(\varphi),\dot{a}(\varphi) \dot{\gamma}(\varphi)^{\bot}
    + a(\varphi) \ddot{\gamma}(\varphi)^{\bot}
    + \dot{b}(\varphi) \dot{\gamma}(\varphi)
    + b(\varphi) \ddot{\gamma}(\varphi)\rangle\\
    &= a(\varphi) \langle \dot{\gamma}(\varphi),\ddot{\gamma}(\varphi)^{\bot} \rangle + \dot{b}(\varphi)c^2_{\gamma}
    = -a(\varphi) c_{\gamma}^{3} \kappa_{\gamma}(\varphi) + \dot{b}(\varphi) c^2_{\gamma}\; .
  \end{align*}
  Using the initial conditions for $\sigma$, we obtain the initial conditions for $a$ and $b$ and the value of
  $c = c_\gamma^2\dot{b}(0)$.
\end{proof}

We are now able to formulate the main result of this article.
Here and in the following we denote by $\lVert \gamma \rVert_k$ the
$C^k$ norm on the space of curves.
Similarly, if $F\colon C^k \to C^\ell$ is a bounded linear operator,
we denote its operator norm by $\lVert F\rVert_{k,\ell}$,
and we use the same notation for norms of multi-linear operators.

\begin{theorem}
  The circular integral invariant $I_{r} \colon \mathcal{C}_{k+\ell+1} \to C^{k}(S^{1};\mathbb{R})$, $k \geq 1$,
  $\ell \ge 1$ is $\ell$-times continuously Fr\'echet differentiable on a neighborhood 
  $\mathcal{U} \subset \mathcal{C}_{k+\ell+1}$
  of the circle of radius $R > r/2$ and its tangential mapping
  is injective on this neighborhood.
  
  Moreover there exists a neighborhood $\tilde{\mathcal{U}}$ of the circle with respect to the
  topology induced by the $C^{k+6}$-norm and a constant $c > 0$ such that for
  every $\gamma$, $\tilde{\gamma} \in \tilde{\mathcal{U}}$ the stability estimate
  \[
  \lVert I_r[\gamma] - I_r[\tilde{\gamma}]\rVert_k
  \ge c \lVert \gamma - \tilde{\gamma}\rVert_k
  \]
  holds.
  In particular, the mapping $I_r$ is injective on $\mathcal{V}$.
\end{theorem}

\begin{remark}
  The condition on $r$ to be smaller than $2R$ is necessary, because otherwise the circular integral invariant in
  each point $\varphi$ is constant equal to $R^{2}\pi$, the area of the circle, and the same holds for any sufficiently
  small deformation of the circle which preserves the area.
\end{remark}

\section{Fr{\'e}chet Differentiability of the Circular Integral Invariant}\label{se:diff}

In the following, we discuss the differentiability of $I_r$ and derive
an analytic formula for $I_r$ and, under certain smoothness assumptions,
its derivative $I_r'$ valid in a neighborhood of the circle.
As a first step, we recall the following result on the differentiability
of the composition mapping.
To that end, we need the following definitions of differentiability
of mappings on Banach spaces.

\begin{definition}
  Let $X$, $Y$ be Banach spaces, $F\colon X \to Y$, and $\ell > 1$.
  The mapping $F$ is called $\ell$-times \emph{weakly differentiable}, if it is
  $\ell$-times G\^ateaux differentiable and its G\^ateaux differential 
  $d^{\ell}F$ is continuous as a mapping 
  \[
  d^{\ell}F\colon X^{\ell+1} \to Y.
  \]
  In contrast, it is called $\ell$-times \emph{continuously Fr\'echet differentiable}
  or \emph{of class $C^\ell$},
  if it is $\ell$-times G\^ateaux differentiable and $d^{\ell}F$ is continuous
  as a mapping
  \[
  d^{\ell}F \colon X \to L^{\ell}(X,Y).
  \]
  Here $L^{\ell}(X,Y)$ is the Banach space of $\ell$-linear mappings
  from $X^{\ell}$ to $Y$ equipped with the operator norm.
\end{definition}

Note that the continuity requirement for weak differentiability
is weaker than for continuous Fr\'echet differentiability,
and thus a weakly differentiable mapping need not be Fr\'echet
differentiable of the same order.
The following Lemma shows, however, that it is Fr\'echet
differentiable of lower order.

\begin{lemma}\label{le:weakFrechet}
  Let $X$, $Y$ be Banach spaces and $F \colon X \to Y$
  $(\ell+1)$-times weakly differentiable with $\ell \ge 1$.
  Then $F$ is $\ell$-times continuously Fr\'echet differentiable.
\end{lemma}

\begin{proof}
  Let $x \in X$ and $\varepsilon > 0$.
  
  We have to show that there exists $\delta > 0$ such that
  for every $y \in X$ with $\lVert x - y \rVert < \delta$
  the inequality
  \[
  \lVert d^{\ell}F(x)-d^{\ell}F(y)\rVert_{L^{\ell}(X,Y)}
  = \sup_{\substack{z \in X^{\ell}\\ \lVert z \rVert_{X^{\ell}} = 1}}
    \lVert d^{\ell}F(x)(z)-d^{\ell}F(y)(z)\rVert_Y
    < \varepsilon
  \]
  holds.

  Using the continuity of $d^{\ell+1}F\colon X^{\ell+2}\to Y$ at
  the point $(x,0,0) \in X\times X \times X^{\ell}$ and the fact that
  $d^{\ell+1}F(x,0,0) = 0$, we obtain the existence of $\eta > 0$ such that
  \[
  \lVert d^{\ell+1}F(x_1,x_2,\tilde{z})\Bigr\rVert \le \varepsilon
  \qquad\text{whenever}
  \qquad
  \lVert x_1 - x\rVert + \lVert x_2 \rVert + \lVert \tilde{z} \rVert < 3\eta.
  \]
  Now let $\delta := \min\{\eta^{\ell+1},1\}$,
  let $y \in X$ with $\lVert y-x\rVert < \delta$
  and $z \in X^{\ell}$ with $\lVert z \rVert \le 1$.
  Then we have
  \[
  \begin{aligned}
    \lVert d^{\ell}F(x)(z)-d^{\ell}F(y)(z)\rVert_Y
    &= \Bigl\lVert\int_0^1 \partial_t d^{\ell}F(x+t(y-x),z)\,dt\Bigr\rVert\\
    &= \Bigl\lVert\int_0^1 d^{\ell+1}F(x+t(y-x),y-x,z)\,dt\Bigr\rVert\\
    &\le\int_0^1 \Bigl\lVert d^{\ell+1}F(x+t(y-x),y-x,z)\Bigr\rVert\,dt\\
    &=\int_0^1 \Bigl\lVert d^{\ell+1}F\Bigl(x+t(y-x),\frac{y-x}{\eta^\ell},\eta z\Bigr)\Bigr\rVert\,dt < \varepsilon.
  \end{aligned}
  \]
\end{proof}

\begin{lemma}\label{composition}
  For every $k \geq 0$, $\ell \ge 1$ the composition mapping
  \begin{equation*}
    \begin{aligned}
      \comp \colon C^{k+\ell+1}(S^1;\mathbb{R}^2) \times \diff^{k}(S^1)   &\to C^{k}(S^1;\mathbb{R}^2)\,, \\
      (f,g)   &\mapsto f \circ g\,,
    \end{aligned}
  \end{equation*}
  is $(\ell+1)$-times weakly differentiable and therefore
  $\ell$-times continuously Fr{\'e}chet differentiable.
\end{lemma}

\begin{proof}
  The result on weak differentiability has been shown in~\cite[Section~6.9]{Mic06}
  (note that the result in \cite{Mic06} has been shown for the space $HC^{n}$,
  which, however, is equivalent to $C^n$ in the case of a compact manifold).
  The statement concerning the continuous Fr\'echet differentiability
  follows from Lemma~\ref{le:weakFrechet}.  
\end{proof}

\begin{remark}
  In order to simplify the notation, we will sometimes omit the
  domain and range of the function spaces in expressions like $C^k(S^1;\mathbb{R})$
  and write $C^k$ instead, if no confusion is possible.
\end{remark}

\begin{remark}
  We will need two different types of derivatives for the formulation
  of our results: First, Fr\'echet derivatives in the function space $C^k$,
  and, second, derivatives of functions $f \in C^k(S^1)$ with respect to
  their argument $\varphi\in S^1$.
  In order to highlight this difference, we use the following notation:
  For a function $F\colon C^k \to C^j$, we denote by $F'\colon C^k \to L(C^k,C^j)$
  its Fr\'echet derivative.
  In contrast, if $f \in C^k$, then $\dot{f}$ denotes its derivative in the parameter space.
\end{remark}

In order to make the notation less cumbersome, we omit the argument $\varphi$ in $\gamma(\varphi)$,
$\sigma(\varphi)$ and similar expressions if the argument is clear from the context.

\begin{lemma}\label{intersection_parameters}
  For each $k \geq 0$, $\ell \ge 1$ there exists a neighborhood $\mathcal{V} \subset C^{k+\ell+1}(S^1;\mathbb{R}^2)$ of the
  constant speed parameterized circle of radius $R > r/2$ such that the following hold:
  \begin{enumerate}
  \item For each $\gamma \in \mathcal{V}$ and $\varphi \in S^1$ the circle $B_r \big( \gamma(\varphi) \big)$ intersects the curve $\gamma$ in
    exactly two points, denoted by $\gamma\big(p_{\gamma}(\varphi)\big)$ and $\gamma \big( m_{\gamma}(\varphi) \big)$.
    Here $m_{\gamma}(\varphi)$ denotes the previous intersection parameter and $p_{\gamma}(\varphi)$ the next one
    (see Figure \ref{fig:sketch}).
  \item The mappings
    \begin{equation*}
      \begin{aligned}
        m \colon C^{k+\ell+1}(S^1;\mathbb{R}^2)  &\to \diff^{k}(S^{1})\,, &\quad&& p \colon C^{k+\ell+1}(S^1;\mathbb{R}^2) &\to \diff^{k}(S^{1})\,,\\
        \gamma   &\mapsto m_{\gamma}\,,      &&&                             \gamma  &\mapsto p_{\gamma}\,,
      \end{aligned}
    \end{equation*}
    are $\ell$-times continuously Fr{\'e}chet differentiable.
    The first derivatives in direction $\sigma \in C^{k+\ell+1}(S^1;\mathbb{R}^2)$ are given by
    \begin{equation*}
      p_{\gamma}'(\sigma)
      = \frac{\langle \sigma-\sigma(p_\gamma),\gamma(p_\gamma)-\gamma \rangle}{\langle \dot{\gamma}(p_\gamma),\gamma(p_\gamma)-\gamma \rangle} \;,
      \quad
      m_{\gamma}'(\sigma)
      = \frac{\langle \sigma-\sigma(m_\gamma),\gamma(m_\gamma)-\gamma \rangle}{\langle \dot{\gamma}(m_\gamma),\gamma(m_\gamma)-\gamma \rangle} \;.
    \end{equation*}
    Here $\diff^{k}(S^{1})$ denotes the group of $C^{k}$-diffeomorphisms on the unit circle.
  \end{enumerate}
\end{lemma}

\begin{proof}
  Denote by $\gamma_0 \colon S^1 \to \mathbb{R}^2$ the constant speed parameterized
  circle, that is, $\gamma_0(\varphi) = \bigl(R\cos(\varphi),R\sin(\varphi)\bigr)$.
  Then, for every $\varphi \in S^1$, the circle $B_r\bigl(\gamma_0(\varphi)\bigr)$ intersects
  $\gamma_0$ precisely at the two points $\bigl(R\cos(\varphi\pm\vartheta),R\sin(\varphi\pm\vartheta)\bigr)$,
  where
  \[
  \vartheta = \arccos\biggl(1-\frac{r^2}{2R^2}\biggr)\;.
  \]
  Now define the mapping
  \[
  \begin{aligned}
    F \colon C^{k+\ell+1}(S^1;\mathbb{R}^2) \times \diff^k(S^1) &\to C^k(S^1;\mathbb{R})\,,\\
    (\gamma,d) &\mapsto \lvert\gamma(d) - \gamma\rvert^2 - R^2\;.
  \end{aligned}
  \]
  Obviously, the mappings $p_\gamma$ and $m_\gamma$ we are searching for satisfy
  $F(\gamma,p_\gamma) = 0$ and $F(\gamma,m_\gamma) = 0$.
  In particular, the equation $F(\gamma_0,d) = 0$ has the two solutions
  $p_{\gamma_0}(\varphi) := \varphi+\vartheta$ and $m_{\gamma_0}(\varphi) := \varphi-\vartheta$.
  Now, Lemma~\ref{composition} implies that the mapping
  $\comp\colon C^{k+\ell+1}(S^1;\mathbb{R}^2) \times \diff^k(S^1) \to C^k(S^1;\mathbb{R}^2)$, and consequently
  also $F$, is $\ell$-times continuously Fr\'echet differentiable.
  Moreover, it is easy to see that the derivative of $F$ at $(\gamma_0,p_{\gamma_0})$
  in direction $(0,\tau) \in C^{k+\ell+1}(S^1;\mathbb{R}^2) \times C^k(S^1;\mathbb{R})$ is given as
  \[
  F'(\gamma_0,p_{\gamma_0})(0,\tau) = 2\langle\gamma_0(p_{\gamma_0})-\gamma_0,\dot{\gamma}_0(p_{\gamma_0})\rangle\tau
  =2\sin(\vartheta)\tau\,,
  \]
  which is obviously an isomorphism of $C^k(S^1;\mathbb{R})$, as
  the assumption $R > r/2 > 0$ implies that $\sin(\vartheta) \neq 0$.
  Thus the implicit function theorem on Banach spaces
  (see~\cite[Sec.~I.5]{Lan95}) implies
  the existence of a neighborhood $\mathcal{V} \subset C^{k+\ell+1}(S^1;\mathbb{R}^2)$ of $\gamma_0$
  and unique $\ell$-times continuously Fr\'echet differentiable mappings
  $m$,~$p\colon \mathcal{V} \to \diff^k(S^1)$ satisfying the equations
  $F(\gamma,p_\gamma) = 0 = F(\gamma,m_\gamma)$.
  The formula for the directional derivative of $p$ at $\gamma$
  in direction $\sigma$ now follows from the fact that
  \[
  0 = \partial_\gamma F(\gamma,p_\gamma)(\sigma)
  = 2\langle\gamma(p_\gamma)-\gamma,\sigma(p_\gamma)-\sigma\rangle+
  2p'_\gamma(\sigma) \langle\gamma(p_\gamma)-\gamma,\dot{\gamma}(p_\gamma)\rangle\,,
  \]
  which is a simple application of the chain rule.
  The formula for $m_\gamma'(\sigma)$ can be derived analogously.
\end{proof}

\begin{theorem}\label{Variation_general}
  For each $k \geq 0$ and $\ell\ge 1$ there exists a neighborhood 
  $\mathcal{V} \subset C^{k+\ell+1}(S^1;\mathbb{R}^2)$ of the circle of
  radius $R > r/2$ such that the following hold:
  \begin{enumerate}
  \item   For each $\gamma \in \mathcal{V}$ the circular integral invariant $I_{r}[\gamma]$ can be written as
    \begin{equation} \label{eq:circ_invariant_reformulation}
      I_{r}[\gamma](\varphi)
      =
      \frac{1}{2} \int_{m_{\gamma}}^{p_{\gamma}} \langle \gamma(\psi)-\gamma,\dot{\gamma}(\psi)^{\bot} \rangle d\psi
      +\frac{r^2}{2} \arccos \biggl(\frac{\langle \gamma(p_{\gamma})-\gamma,\gamma(m_{\gamma})-\gamma \rangle}{r^2} \biggr) \;.
    \end{equation}
  \item   The circular integral invariant
    \begin{equation*}
      \begin{aligned}
        I_{r} \colon \mathcal{V}  &\to C^{k}(S^{1};\mathbb{R})\,,\\
        \gamma  &\mapsto I_{r}[\gamma]\,,
      \end{aligned}
    \end{equation*}
    is $\ell$-times continuously Fr{\'e}chet differentiable. Its derivative in direction
    $\sigma\in C^{k+\ell+1}(S^1;\mathbb{R}^2)$ is given by
    \begin{align*}
      &2I'_r[\gamma](\sigma)\\
      &=
      2 \int_{m_{\gamma}}^{p_{\gamma}} \langle \sigma(\psi),\dot{\gamma}(\psi)^{\bot} \rangle d\psi \\
      & \qquad
      - \langle \sigma,\gamma(p_{\gamma})^{\bot}-\gamma(m_{\gamma})^{\bot} \rangle
      + \langle \gamma(p_{\gamma})-\gamma,\sigma(p_{\gamma})^{\bot} \rangle
      - \langle \gamma(m_{\gamma})-\gamma,\sigma(m_{\gamma})^{\bot} \rangle  \\
      & \qquad
      + \langle \gamma(p_{\gamma})-\gamma,\dot{\gamma}(p_{\gamma})^{\bot} \rangle
      \frac{\langle \sigma-\sigma(p_{\gamma}),\gamma(p_{\gamma})-\gamma \rangle}{\langle \dot{\gamma}(p_{\gamma}),\gamma(p_{\gamma})-\gamma \rangle}\\
      & \qquad
      - \langle \gamma(m_{\gamma})-\gamma,\dot{\gamma}(m_{\gamma})^{\bot}\rangle
      \frac{\langle \sigma-\sigma(m_{\gamma}),\gamma(m_{\gamma})-\gamma \rangle}{\langle \dot{\gamma}(m_{\gamma}),\gamma(m_{\gamma})-\gamma \rangle}\\
      & \qquad
      - \frac{r^2}{\sqrt{r^4 - \langle \gamma(p_{\gamma})-\gamma,\gamma(m_{\gamma})-\gamma\rangle^2}}\cdot\\
      & \qquad \qquad
      \cdot\bigg( \!
      \frac{\langle \sigma-\sigma(p_{\gamma}),\gamma(p_{\gamma})-\gamma \rangle}{\langle \dot{\gamma}(p_{\gamma}),\gamma(p_{\gamma})-\gamma \rangle}
      \langle \dot{\gamma}(p_{\gamma}),\gamma(m_{\gamma})-\gamma \rangle \\
      & \qquad \qquad \qquad
      + \langle \sigma(p_{\gamma})-\sigma,\gamma(m_{\gamma})-\gamma \rangle
      + \langle \gamma(p_{\gamma})-\gamma,\sigma(m_{\gamma})-\sigma \rangle\\
      & \qquad \qquad \qquad
      + \frac{\langle \sigma-\sigma(m_{\gamma}),\gamma(m_{\gamma})-\gamma \rangle}{\langle \dot{\gamma}(m_{\gamma}),\gamma(m_{\gamma})-\gamma \rangle}
      \langle \gamma(p_{\gamma})-\gamma,\dot{\gamma}(m_{\gamma})\rangle
      \! \bigg)\;.
    \end{align*}
  \end{enumerate}	
\end{theorem}

\begin{proof}
  Under the given assumptions for $R$, $r$ and $\mathcal{V}$, 
  Formula~\eqref{eq:circ_invariant_reformulation}  can be easily deduced from Figure~\ref{fig:sketch}.
  A term by term investigation of Formula~\eqref{eq:circ_invariant_reformulation},
  using Lemma~\ref{composition} and Lemma~\ref{intersection_parameters}, shows that
  $I_r$ is of class $C^\ell$ on $\mathcal{V}$.
  \begin{figure}[H]
    \centering
    \def\svgwidth{0.8\textwidth}
    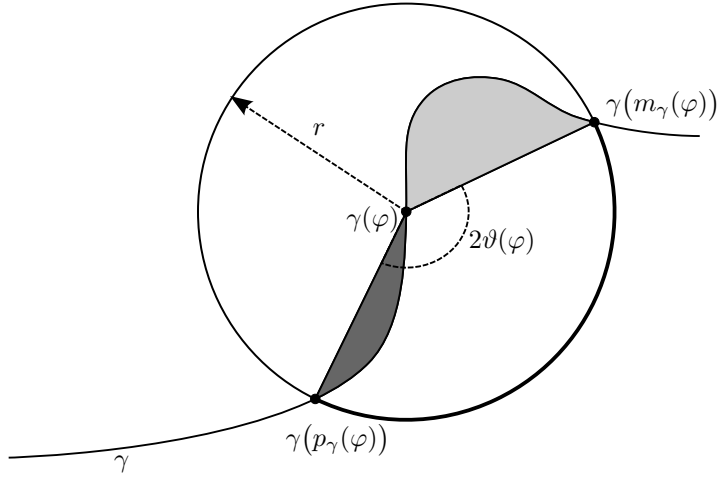
    \caption{Sketch of the derivation of the analytical formula for the circular integral invariant
      assuming two points of intersection.}
    \label{fig:sketch}
  \end{figure}
  
  To calculate the differential of $I_r$ we treat the two terms of Formula~\eqref{eq:circ_invariant_reformulation} separately.
  For the first term we obtain
  \begin{align*}
    \partial_{\gamma}
    &
    \int_{m_{\gamma}}^{p_{\gamma}}
    \langle \gamma(\psi)-\gamma,\dot{\gamma}(\psi)^{\bot} \rangle
    d\psi \\
    &=
    \int_{m_{\gamma}}^{p_{\gamma}}
    \langle \sigma(\psi)-\sigma,\dot{\gamma}(\psi)^{\bot} \rangle
    + \langle \gamma(\psi)-\gamma,\dot{\sigma}(\psi)^{\bot} \rangle
    d\psi\\
    & \qquad
    + \langle \gamma(p_{\gamma})-\gamma,\dot{\gamma}(p_{\gamma})^{\bot} \rangle
    p'_{\gamma}(\sigma)
    - \langle \gamma(m_{\gamma})-\gamma,\dot{\gamma}(m_{\gamma})^{\bot} \rangle
    m'_{\gamma}(\sigma)\\
    &=
    \int_{m_{\gamma}}^{p_{\gamma}} \langle \sigma(\psi)-\sigma,\dot{\gamma}(\psi)^{\bot} \rangle d\psi
    - \int_{m_{\gamma}}^{p_{\gamma}} \langle \dot{\gamma}(\psi),\sigma(\psi)^{\bot} \rangle d\psi \\
    & \qquad
    + \langle \gamma(p_{\gamma})-\gamma,\sigma(p_{\gamma})^{\bot} \rangle
    - \langle \gamma(m_{\gamma})-\gamma,\sigma(m_{\gamma})^{\bot} \rangle\\
    & \qquad
    + \langle \gamma(p_{\gamma})-\gamma,\dot{\gamma}(p_{\gamma})^{\bot} \rangle p'_{\gamma}(\sigma)
    - \langle \gamma(m_{\gamma})-\gamma,\dot{\gamma}(m_{\gamma})^{\bot} \rangle m'_{\gamma}(\sigma)\\
    &=
    \int_{m_{\gamma}}^{p_{\gamma}} 2 \langle \sigma(\psi),\dot{\gamma}(\psi)^{\bot} \rangle d\psi
    - \langle \sigma,\gamma(p_{\gamma})^{\bot}-\gamma(m_{\gamma})^{\bot} \rangle\\
    & \qquad
    + \langle \gamma(p_{\gamma})-\gamma,\sigma(p_{\gamma})^{\bot} \rangle
    - \langle \gamma(m_{\gamma})-\gamma,\sigma(m_{\gamma})^{\bot} \rangle\\
    & \qquad
    + \langle \gamma(p_{\gamma})-\gamma,\dot{\gamma}(p_{\gamma})^{\bot} \rangle
    p'_{\gamma}(\sigma)
    - \langle \gamma(m_{\gamma})-\gamma,\dot{\gamma}(m_{\gamma})^{\bot} \rangle
    m'_{\gamma}(\sigma) \;.
  \end{align*}
  A simple application of the chain rule yields for the second term
  \begin{align*}
    &-\partial_{\gamma} \frac{r^2}{2} \arccos
    \bigg( \!
    \frac{\langle \gamma(p_{\gamma})-\gamma,
      \gamma(m_{\gamma})-\gamma\rangle
    }{r^2}
    \! \bigg)\\
    &=
    \frac{\langle \dot{\gamma}(p_{\gamma}) p'_{\gamma}(\sigma) + \sigma(p_{\gamma}) - \sigma,
      \gamma(m_{\gamma})-\gamma
      \rangle
      +
      \langle \gamma(p_{\gamma})-\gamma,
      \dot{\gamma}(m_{\gamma}) m'_{\gamma}(\sigma) + \sigma(m_{\gamma}) - \sigma
      \rangle}
	 {2r^{-2} \sqrt{r^4 - \langle \gamma(p_{\gamma})-\gamma,\gamma(m_{\gamma})-\gamma \rangle^2}} \,.
  \end{align*}
  Using the formulas for the intersection parameters, we obtain the desired result.
\end{proof}

In the special case where $\gamma$ equals the unit circle the lemma above reduces to:

\begin{lemma}\label{FormelKreis}
  Let $\gamma \in C^{k+\ell+1}(S^1;\mathbb{R}^2)$, $k \geq 0$, $\ell \ge 1$,
  be the constant speed parameterized unit circle, that is,
  \begin{equation*}
    \gamma(\varphi) = \big( \!\cos(\varphi),\sin(\varphi) \big) \;,
  \end{equation*}
  and let $r<2$. Then the derivative of $I_{r}[\gamma]$ in direction $\sigma \in C^{k+\ell+1}(S^1;\mathbb{R}^2)$ with
  \begin{equation*}
    \sigma(\varphi) = a(\varphi) \dot{\gamma}(\varphi)^{\bot} +b(\varphi) \dot{\gamma}(\varphi)
  \end{equation*}
  is given by
  \begin{equation*}
    I'_{r}[\gamma](\sigma)(\varphi)
    =
    \int_{\varphi-\vartheta}^{\varphi+\vartheta} a(\psi) \, d\psi - 2 \sin(\vartheta) a(\varphi)
    =
    \big( \chi_{[-\vartheta,\vartheta]} \ast a \big)(\varphi) - 2 \sin(\vartheta) a(\varphi)
  \end{equation*}
  with
  \begin{equation*}
    \vartheta := \arccos \bigg( \! 1 - \frac{r^2}{2} \bigg) \;.
  \end{equation*}
\end{lemma}
The proof of this lemma is postponed to the appendix.

\section{Proof of the Main Theorem}

\begin{proof}
  We have already shown in Section~\ref{se:diff} that
  $I_r$ is of class $C^\ell$.

  We now show the local injectivity of $I_r'$.
  Without loss of generality we may assume that $R = 1$.
  Let $\mathcal{V}$ be the neighborhood of the unit circle
  defined in Theorem~\ref{Variation_general}.
  The formula for $I_r'$ implies that
  for every $\gamma \in \mathcal{V} \subset C^{k+\ell+1}(S^1;\mathbb{R}^2)$
  the mapping $I_r'[\gamma]$ is bounded as a mapping
  from $C^k(S^1;\mathbb{R}^2)$ to $C^k(S^1;\mathbb{R})$.
  Thus, $I_r'[\gamma]$ has a unique bounded extension
  $J_r[\gamma] \colon C^k(S^1;\mathbb{R}^2) \to C^k(S^1;\mathbb{R})$.
  Moreover, the mapping $J_r$
  is continuous with respect to the $C^{k+\ell+1}$-topology
  seen as a mapping from $\mathcal{V}$ to $L\bigl(C^k(S^1;\mathbb{R}^2),C^k(S^1;\mathbb{R})\bigr)$.
  In addition, we denote for $\gamma \in \mathcal{V}$ by $\tilde{J}_r[\gamma]$
  the restriction of $J_r[\gamma]$ to $T_\gamma\mathcal{C}_k$.
  
  Denote now by $\gamma_0\colon S^1 \to \mathbb{R}^2$ the
  constant speed parameterized unit circle.
  Define for given
  $\sigma \in T_{\gamma_0}\mathcal{C}_{k}$ the function $A\sigma \colon S^1\to \mathbb{R}$ by
  \begin{equation*}
    A \sigma(\varphi) := \langle \sigma(\varphi),\dot{\gamma}_0(\varphi)^{\bot} \rangle \;.
  \end{equation*}
  Because $\gamma_0$ is a $C^\infty$-curve, it follows that $A\sigma$ is $C^{k}$. Using Lemma
  \ref{lem:tangent_space_characterization} it follows that $A \sigma(0) = 0$ and $\partial_{\varphi}(A \sigma)(0) = 0$.
  Define now the space
  \[
  \mathcal{A}^{k}(S^1;\mathbb{R}) :=
  \bigl\{ a \in C^k(S^1;\mathbb{R}) : a(0) = \dot{a}(0) = 0\bigr\}.
  \]
  Then it follows that $A$ is a bounded linear mapping from $T_{\gamma_0} \mathcal{C}_{k}$ to $\mathcal{A}^k(S^{1};\mathbb{R})$.
  
  In addition, it follows from Lemma \ref{lem:tangent_space_characterization} that $A$ is boundedly invertible
  with $A^{-1}$ given by
  \begin{equation*}
    A^{-1} a = a \dot{\gamma}_0^{\bot} + b \dot{\gamma}_0
  \end{equation*}
  with
  \begin{equation*}
    b(\varphi) = \dot{b}(0) \varphi + \int_{0}^{\varphi} a(\tau) \, d\tau
    \quad \text{ and } \quad
    \dot{b}(0) = \frac{1}{2\pi} \int_{0}^{2\pi} a(\tau) \, d\tau \;.
  \end{equation*}
  The expression for $\dot{b}(0)$ is due to the periodicity of $b$, which implies that
  \begin{equation*}
    0 = b(0) = b(2\pi) = 2\pi \dot{b}(0) + \int_{0}^{2\pi} a(\tau) \, d\tau \; .
  \end{equation*}
  Therefore $A$ is in fact an isomorphism between $T_{\gamma_0} \mathcal{C}_{k}$ and $\mathcal{A}^k(S^{1};\mathbb{R})$.
  
  According to Lemma~\ref{FormelKreis}, the mapping $\tilde{J}_r[\gamma_0]$
  evaluated at
  $\sigma = a \dot{\gamma}_0^{\bot} + b \dot{\gamma}_0 \in T_{\gamma_0}\mathcal{C}_k$ can be written as
  \begin{equation*}
    \tilde{J}_{r}[\gamma_0](\sigma) = \chi_{[-\vartheta,\vartheta]} \ast a - 2 \sin(\vartheta) a \; .
  \end{equation*}
  Thus $\tilde{J}_r[\gamma_0]$ can be decomposed into
  \begin{equation*}
    \tilde{J}_{r}[\gamma_0] = B \circ \imath \circ A \;,
  \end{equation*}
  where the operator $B \colon C^{k}(S^{1};\mathbb{R}) \to C^{k}(S^{1};\mathbb{R})$ is given by
  \begin{equation*}
    B a = \chi_{[-\vartheta,\vartheta]} \ast a - 2 \sin(\vartheta) a
  \end{equation*}
  and $\imath$ is the embedding from $\mathcal{A}^k(S^{1};\mathbb{R})$ into $C^{k}(S^{1};\mathbb{R})$.
  Lemma~\ref{le:compact} (see Appendix) implies that the mapping
  $\sigma \mapsto \chi_{[-\vartheta,\vartheta]}\ast a$ is compact and thus $B$ is a compact
  perturbation of the identity. Therefore the Riesz--Schauder theory (see~\cite[Chap.~X.5]{Yos95}) implies that
  $B$ has a closed range.
  
  Next we compute the kernel of $B$. To that end we consider the mapping in the Fourier basis. A short calculation
  shows that in this basis the operator $B$ is the diagonal operator that maps a sequence of (complex) Fourier
  coefficients $(c_{j})_{j \in \mathbb{Z}}$ to the sequence $(d_{j} c_{j})_{j\in\mathbb{Z}}$,
  where
  \begin{equation*}
    d_j = \begin{cases}
      2 \bigl( 1 - \sin(\vartheta) \bigr) & \text{ if } j = 0 \,,\\
      0 & \text{ if } j = \pm 1 \,,\\
      \Bigl[ 2 \frac{\sin(j\vartheta)}{j} - 2 \sin(\vartheta) \Bigr] & \text{ else.}
    \end{cases}
  \end{equation*}
  Because $\sin(\vartheta) \neq 1$ and $\sin(j \vartheta) \neq j \sin(\vartheta)$ whenever
  $j \in \mathbb{Z} \setminus \{-1,0,1\}$ (see Lemma \ref{lem:sine_inequality} in the Appendix), it follows that
  the kernel of $B$ consists of the functions $a$ of the form $a(\varphi) = c_{-1} \exp(-i\varphi) + c_{1} \exp(i\varphi)$
  for some $c_{-1}$, $c_{1} \in \mathbb{C}$.
  
  In the next step we show that the kernel of $\tilde{J}_r[\gamma_0] = B \circ \imath \circ A$ is trivial. Therefore assume that
  $a = c_{-1} \exp(-i\cdot) + c_{1} \exp(i\cdot) \in \mathcal{A}^k(S^{1};\mathbb{R}) \cap \kernel B$. Because $a(0) = 0$,
  it follows that $c_{-1} + c_{1} = 0$; because $\dot{a}(0) = 0$, it follows that $-c_{-1} + c_{1} = 0$. Together, this
  shows that $c_{-1} = c_{1} = 0$, implying that the intersection of $\kernel B$ with $\mathcal{A}^k(S^{1};\mathbb{R})$ is
  trivial. Since $A$ is an isomorphism this proves the injectivity of $\tilde{J}_{r}[\gamma_0]$.
  
  We have thus shown that $\tilde{J}_{r}[\gamma_0] = B \circ \imath \circ A$ is strongly closed and injective.
  Now note that the set of strongly closed and injective, bounded linear functionals
  between two Banach spaces $X$ and $Y$ is open with respect to the
  norm topology on $L(X,Y)$.
  Because of the continuity of $J_r$ and therefore $\tilde{J}_r$,
  this proves the existence of a neighborhood $\gamma_0 \in \mathcal{U} \subset \mathcal{C}_{k+2}$
  such that $\tilde{J}_r[\gamma]$ is injective for every $\gamma \in \mathcal{U}$.
  In particular, this proves the injectivity of $I_r'[\gamma]$ for every $\gamma\in\mathcal{U}$.

  For proving the local injectivity, let $\gamma$, $\tilde{\gamma} \in \mathcal{C}_{k+6} \cap \mathcal{U}$.
  Because the mapping $I_r\colon \mathcal{V}\subset C^{k+3}(S^1;\mathbb{R}^2) \to C^k(S^1;\mathbb{R})$
  is of class $C^2$, it has a Taylor expansion of the form
  \[
  I_r[\tilde{\gamma}] = I_r[\gamma] + I_r'[\gamma](\tilde{\gamma}-\gamma)
  + \int_0^1(1-t)I_r''[\gamma+t(\tilde{\gamma}-\gamma)](\tilde{\gamma}-\gamma,\tilde{\gamma}-\gamma)\,dt.
  \]
  Thus
  \[
  \begin{aligned}
    \lVert I_r[\tilde{\gamma}]-I_r[\gamma]\rVert_k
    &\ge \lVert I_r'[\gamma](\tilde{\gamma}-\gamma)\rVert_k
    - \Bigl\rVert\int_0^1(1-t)I_r''[\gamma+t(\tilde{\gamma}-\gamma)](\tilde{\gamma}-\gamma,\tilde{\gamma}-\gamma)\,dt\Bigr\rVert_k\\
    &\ge \lVert I_r'[\gamma](\tilde{\gamma}-\gamma)\rVert_k
    - \int_0^1\Bigl\rVert I_r''[\gamma+t(\tilde{\gamma}-\gamma)](\tilde{\gamma}-\gamma,\tilde{\gamma}-\gamma)\Bigr\rVert_k\,dt
  \end{aligned}
  \]
  Because $I_r\colon C^{k+3}\to C^k$ is of class $C^2$, it follows that 
  $I_r''\colon C^{k+3} \to L^2(C^{k+3},C^k)$ is continuous.
  Thus there exists a convex neighborhood $\mathcal{V}_1$ of the circle with
  respect to the $C^{k+3}$-norm and a constant $c_1$ such that
  $\lVert I_r''[\widehat{\gamma}]\rVert_{L^2(C^{k+3},C^k)} \le c_1$ for every $\widehat{\gamma} \in \mathcal{V}_1$.
  Consequently,
  \begin{equation}\label{eq:1}
    \lVert I_r[\tilde{\gamma}]-I_r[\gamma]\rVert_k
    \ge \lVert I_r'[\gamma](\tilde{\gamma}-\gamma)\rVert_k - c_1 \lVert \tilde{\gamma}-\gamma\rVert_{k+3}^2
  \end{equation}
  for every $\gamma$, $\tilde{\gamma} \in \mathcal{V}_1$.
  
  Since $I_r'\colon C^{k+3} \to L(C^{k+3},C^k)$ is continuous,
  for every $\varepsilon > 0$  there exists
  a neighbourhood $\mathcal{V}_\varepsilon$ of the
  constant speed parameterized unit circle $\gamma_0$
  such that $\lVert I_r'[\gamma] - I_r'[\gamma_0]\rVert < \varepsilon$
  for every $\gamma \in \mathcal{V}_\varepsilon$.
  In particular, we have for $\gamma \in \mathcal{V}_\varepsilon$
  \begin{equation}\label{eq:2}
  \lVert I_r'[\gamma](\tilde{\gamma}-\gamma)\rVert_k
  \ge \lVert I_r'[\gamma_0](\tilde{\gamma}-\gamma)\rVert_k
  - \varepsilon\lVert \tilde{\gamma}-\gamma\rVert_k.
  \end{equation}

  Since $\mathcal{C}_{k}$ is a smooth submanifold of $C^{k}(S^1;\mathbb{R}^2)$,
  there exists a neighborhood $\mathcal{V}_2 \subset C^k(S^1;\mathbb{R}^2)$ of $\gamma_0$ and
  a smooth diffeomorphism 
  $\Phi\colon \mathcal{V}_2 \to \Phi(\mathcal{V}_2)\subset C^k(S^1;\mathbb{R}^2)$
  such that 
  \[
  \Phi[\gamma_0] = 0,
  \qquad
  \Phi(\mathcal{V}_2 \cap \mathcal{C}_k) \subset T_{\gamma_0}\mathcal{C}_k,
  \quad\text{ and }\quad
  \Phi'[\gamma_0] = \id.
  \]
  In order to show that such a map exists,
  let $\Psi\colon\mathcal{V}_2\subset C^k(S^1;\mathbb{R}^2) \to C^k(S^1;\mathbb{R}^2)$ be
  any submanifold chart centered at $\gamma_0$.
  That is, there exists a closed linear subspace $\mathcal{E} \subset C^k(S^1;\mathbb{R}^2)$
  such that $\Psi(\mathcal{V}_2 \cap \mathcal{C}_k) = \Psi(\mathcal{V}_2) \cap \mathcal{E}$.
  Then the mapping $\Phi := \Psi'[\gamma_0]^{-1} \circ \Psi$ has the desired properties.

  Thus the continuous invertibility of $I_r'[\gamma_0]$ on $T_{\gamma_0}\mathcal{C}_k$
  seen as a mapping from $C^k$ to $C^k$ implies that
  there exists $c_2 > 0$ such that for every $\gamma$, $\tilde{\gamma} \in \mathcal{V}_2 \cap \mathcal{C}_k$
  we have
  \begin{equation}\label{eq:7}
  \begin{aligned}
    \lVert I_r'[\gamma_0](\tilde{\gamma}-\gamma)\rVert_k
    &\ge \lVert I_r'[\gamma_0](\Phi(\tilde{\gamma})-\Phi(\gamma))\rVert_k-
    \lVert I_r'[\gamma_0](\tilde{\gamma}-\gamma-\Phi(\tilde{\gamma})+\Phi(\gamma))\rVert_k\\
    &\ge c_2 \lVert \Phi(\tilde{\gamma})-\Phi(\gamma)\rVert_k
    - c_3 \lVert \tilde{\gamma}-\gamma-\Phi(\tilde{\gamma})+\Phi(\gamma)\rVert_k
  \end{aligned}
  \end{equation}
  with $c_3 := \lVert I_r'[\gamma_0]\rVert_{k+3,k}$.

  Developing $\Phi(\tilde{\gamma})$ in a Taylor expansion
  at $\gamma$, we obtain after possibly choosing a smaller neighbourhood
  \[
  \Phi(\tilde{\gamma}) = \Phi(\gamma) + \Phi'[\gamma](\tilde{\gamma}-\gamma) + R(\gamma,\tilde{\gamma})
  \]
  with
  \[
  \lVert R(\gamma,\tilde{\gamma})\rVert_k \le c_4 \lVert \gamma-\tilde{\gamma}\rVert_k^2
  \]
  for some $c_4 > 0$ independent of $\gamma$.
  Inserting this Taylor expansion into~\eqref{eq:7}
  yields
  \[
  \begin{aligned}
    \lVert I_r'[\gamma_0](\tilde{\gamma}-\gamma)\rVert_k
    &\ge c_2\lVert \Phi'[\gamma](\tilde{\gamma}-\gamma)\rVert_k - 
      c_3\lVert (\id-\Phi'[\gamma])(\tilde{\gamma}-\gamma)\rVert_k\\
    &\qquad\qquad
      - (c_2+c_3)\lVert R(\gamma,\tilde{\gamma})\rVert_k\\
    &\ge c_2\lVert \tilde{\gamma}-\gamma\rVert_k
      - (c_2+c_3)\lVert (\id-\Phi'[\gamma])(\tilde{\gamma}-\gamma)\rVert_k\\
    &\qquad\qquad
      - c_4(c_2+c_3)\lVert \gamma-\tilde{\gamma}\rVert_k^2\\
    &\ge \bigl(c_2- c_5\lVert \id-\Phi'[\gamma]\rVert_{k,k} - c_6\lVert \gamma-\tilde{\gamma}\rVert_k\bigr)
      \lVert \gamma-\tilde{\gamma}\rVert_k.
  \end{aligned}
  \]
  Because $\Phi$ is smooth and $\Phi'[\gamma_0] = \id$,
  it follows that there exists a neighbourhood $\mathcal{V}_3$ of $\gamma_0$
  such that
  \begin{equation}\label{eq:3}
  \lVert I_r'[\gamma_0](\tilde{\gamma}-\gamma)\rVert_k \ge c_7 \lVert \gamma-\tilde{\gamma}\rVert_k
  \end{equation}
  for every $\gamma$, $\tilde{\gamma} \in \mathcal{V}_3 \cap \mathcal{C}_k$.

  Collecting the inequalities~\eqref{eq:1}, \eqref{eq:2}, and~\eqref{eq:3},
  we obtain
  \begin{equation}\label{eq:4}
  \lVert I_r[\tilde{\gamma}]-I_r[\gamma]\rVert_k
  \ge (c_7-\varepsilon)\lVert \gamma-\tilde{\gamma}\rVert_k - c_1\lVert\gamma-\tilde{\gamma}\rVert_{k+3}^2
  \end{equation}
  for every $\gamma$, 
  $\tilde{\gamma} \in \mathcal{V}_1 \cap \mathcal{V}_\varepsilon \cap \mathcal{V}_2 \cap \mathcal{C}_{k+3}$.

  In order to obtain the desired result, we use the interpolation inequality
  (see~\cite[Theorem~2.2.1, p.~143]{Ham82})
  \[
  \lVert\gamma-\tilde{\gamma} \rVert_{k+3}^2 
  \le c_8\lVert\gamma-\tilde{\gamma}\rVert_k \lVert\gamma-\tilde{\gamma}\rVert_{k+6}.
  \]

  Choosing $\varepsilon > 0$ in~\eqref{eq:4} sufficiently small,
  we obtain the estimate
  \[
  \lVert I_r[\tilde{\gamma}]-I_r[\gamma]\rVert_k
  \ge (c_9-c_8\lVert\gamma-\tilde{\gamma}\rVert_{k+6})\lVert\gamma-\tilde{\gamma}\rVert_k
  \]
  for every $\gamma$, $\tilde{\gamma} \in \mathcal{V}_1 \cap \mathcal{V}_\varepsilon \cap \mathcal{V}_2 \cap \mathcal{C}_{k+6}$.
  Thus there exists a neighbourhood $\tilde{\mathcal{U}} \subset \mathcal{C}_{k+6}$
  of the circle $\gamma_0$ and a constant $C > 0$ such that
  \[
  \lVert I_r[\tilde{\gamma}]-I_r[\gamma]\rVert_k
  \ge C\lVert\gamma-\tilde{\gamma}\rVert_k
  \]
  for every $\gamma$, $\tilde{\gamma} \in \tilde{\mathcal{U}}$.
\end{proof}

\section{Conclusion}

In this article, we have shown an injectivity result for the circular integral
invariant on a $C^{k+6}$ neighborhood $\tilde{\mathcal{U}}$ of the circle.
Note, however, that the derived result does not prove the 
continuous invertibility of the invariant on $\tilde{\mathcal{U}}$.

The classical approach for proving such a result
would be the usage of the inverse function theorem.
To that end, however, we would require
that the mapping $I_r$ was continuously Fr\'echet differentiable and its derivative
$I_r'$ an isomorphism of the corresponding tangent spaces.
Although our results prove that $I_r'$ can be extended to an isomorphism $\tilde{J}_r$ of $T_{\gamma_0}\mathcal{C}_k$,
we cannot use the inverse function theorem, as the mapping $I_r$ is not
Fr\'echet differentiable (and not even G\^ateaux differentiable)
from $C^k$ to $C^k$ --- in fact, our results do not even prove
that $I_r$ maps $C^k$ curves into $C^k$ integral invariants.
Conversely, seen as a mapping from $C^{k+2}$ to $C^k$,
the tangent mapping, though injective, cannot be a surjection near the circle.

Another issue are the rather stringent smoothness assumptions.
We believe that it is possible to relax these assumptions,
as it seems probable that $I_r$ is twice weakly differentiable
as a mapping from $C^{k+2}$ to $C^k$, which would indicate
a possible uniqueness result on a $C^{k+4}$-neighborhood of the circle.

\section*{Acknowledgment}

The authors want to thank Peter Michor and G{\"u}nther H{\"o}r\-mann for their comments and suggestions, which helped
to improve the article.
We thank the referee for the careful reading of the article and for pointing
out a mistake in the proof of the original main theorem.

\section{Appendix}

\begin{lemma}\label{le:compact}
  For every $k \ge 0$,
  the mapping $a \mapsto K_\vartheta a := \chi_{[-\vartheta,\vartheta]} \ast a$ is compact as a mapping from
  $C^{k}(S^{1};\mathbb{R})$ to $C^{k}(S^{1};\mathbb{R})$.
\end{lemma}

\begin{proof}
  Denoting by $B_{1}$ the unit ball in $C^{k}(S^{1};\mathbb{R})$, we have to show that the image of $B_{1}$ under
  $K_{\vartheta}$ is precompact in $C^{k}(S^{1};\mathbb{R})$. Applying the Arzel\`a--Ascoli Theorem
  (see~\cite[Chap.~III.3]{Yos95}), we have to show that $K_{\vartheta}(B_{1})$ is bounded and the
  first $k$ derivatives
  of the functions in $K_{\vartheta}(B_{1})$ are equicontinuous. The boundedness of $K_{\vartheta}(B_{1})$ is obvious,
  $K_{\vartheta}$ being a bounded linear mapping (of norm $2 \vartheta$). Now assume that $a \in B_{1}$. Then
  $\partial_{\varphi}^j(K_{\vartheta} a)(\varphi) = a^{(j-1)}(\varphi+\vartheta) - a^{(j-1)}(\varphi-\vartheta)$. Because
  $\lVert a^{(j)} \rVert_{\infty} \le 1$ for all $1 \le j \le k$,
  it follows that $\partial_{\varphi}(K_{\vartheta} a)$ is Lipschitz continuous
  with Lipschitz constant at most 2. Hence $K_{\vartheta}(B_{1})$ is a precompact set.
\end{proof}

\begin{lemma} \label{lem:sine_inequality}
  Let $0 < \vartheta < \pi$ and $j \in \mathbb{Z} \setminus \{-1,0,1\}$. Then $\sin(j\vartheta) \neq j\sin(\vartheta)$.
\end{lemma}

\begin{proof}
  Assume first that $0 < \vartheta \le \pi/2$. We show that in this case the equation
  $\sin(s) = s \sin(\vartheta)/ \vartheta$ has the only solutions $s = 0$ and $s = \pm \vartheta$. First note that
  the strict concavity of the sine function on the interval $[0,\pi]$ implies that on this interval we only have two
  solutions, namely $0$ and $\vartheta$. Moreover, the concavity of the sine implies that
  $\sin(\vartheta)/\vartheta \ge \sin(\pi/2)/(\pi/2) = 2/\pi$, and therefore $\pi\sin(\vartheta)/\vartheta \ge 2$.
  This, however, implies that the equation $\sin(s) = s\sin(\vartheta)/\vartheta$ cannot have any solutions for
  $s > \pi$, as the right hand side is strictly larger than $2$. The fact that $-\vartheta$ is the only negative
  solution follows by symmetry. In particular, setting $s = j\vartheta$, this proves the assertion in the case
  $0 < \vartheta \le \pi/2$.
  
  Now assume that $\pi/2 < \vartheta < \pi$ and let $\psi := \pi - \vartheta$. Then
  \begin{equation*}
    \sin(\vartheta) = \sin(\pi-\psi) = \sin(\psi)\;.
  \end{equation*}
  Now, if $j$ is odd, then
  \begin{equation*}
    \sin(j\vartheta) = \sin(j\pi - j\psi) = \sin(\pi-j\psi) = \sin(j\psi)\;.
  \end{equation*}
  Thus $j\sin(\vartheta) = \sin(j\vartheta)$, if and only if $j\sin(\psi) = \sin(j\psi)$. Because $0 < \psi < \pi/2$,
  the first part of the proof can be applied, showing that $j\sin(\vartheta) \neq \sin(j\vartheta)$ unless $j = \pm 1$.
  
  On the other hand, if $j$ is even, we have
  \begin{equation*}
    \sin(j\vartheta) = \sin(j\pi - j\psi) = \sin(-j\psi) = -\sin(j\psi)\;.
  \end{equation*}
  Thus, $j\sin(\vartheta) = \sin(j\vartheta)$, if and only if $j\sin(\psi) = -\sin(j\psi)$. Now note that the equation
  $-\sin(s) = s\sin(\psi)/\psi$ has only the trivial solution $s = 0$, because $\sin(\psi)/\psi \ge 2/\pi$ and the left
  hand side is negative for $0 < s < \pi$. As a consequence, the equation $j\sin(\psi) = -\sin(j\psi)$ only holds for
  $j = 0$, which concludes the proof.
\end{proof}

\subsection{Proof of Lemma~\ref{FormelKreis}}

\begin{proof}
  Let $\gamma\in C^{k+2}(S^1;\mathbb{R}^2)$ be the constant speed parameterized unit circle. Then
  \begin{align*}
    \gamma(\varphi)					&= \big( \! \cos(\varphi),\sin(\varphi) \big) \,,
    &
    \dot{\gamma}(\varphi)			&= \big( \!-\sin(\varphi),\cos(\varphi)\big)\,,\\
    \gamma(\varphi)^{\bot}			&= \big( \!\sin(\varphi),-\cos(\varphi)\big)\,,
    &
    \dot{\gamma}(\varphi)^{\bot}	&= \big( \!\cos(\varphi),\sin(\varphi)\big)\,.
  \end{align*}
  Because $\gamma$ is the unit circle, there exists $\vartheta \in S^{1}$ such that
  \begin{equation*}
    p_{\gamma}(\varphi) = \varphi+\vartheta \,, \qquad \qquad m_{\gamma}(\varphi) = \varphi-\vartheta\;.
  \end{equation*}
  Obviously, all assumptions of Lemma~\ref{Variation_general} are satisfied. It remains to calculate all the
  terms that appear in the expression of the Fr\'echet derivative in Lemma~\ref{Variation_general} for the special
  case of the unit circle. In particular, we obtain
  \begin{equation*}
    \begin{aligned}
      r^2 = \lVert \gamma(p_{\gamma})-\gamma \rVert^2					&= 2\big(1-\cos(\vartheta)\big)\,,\\
      \langle \gamma(p_{\gamma})-\gamma,\gamma(m_{\gamma})-\gamma \rangle		&=-2\cos(\vartheta)\big(1-\cos(\vartheta)\big)\,,\\
      \langle \dot{\gamma}(p_{\gamma}),\gamma(m_{\gamma})-\gamma \rangle		&= \sin(\vartheta)\big(1-2\cos(\vartheta)\big)\,,\\
      \langle \dot{\gamma}(p_{\gamma})^{\bot},\gamma(p_{\gamma})-\gamma \rangle	&= 1-\cos(\vartheta)\,,\\
      \langle \dot{\gamma}(p_{\gamma}),\gamma(p_{\gamma})-\gamma \rangle		&= \sin(\vartheta)\,,\\
      \langle \dot{\gamma},\gamma(p_{\gamma})-\gamma \rangle			&= \sin(\vartheta)\,,\\
      \langle \dot{\gamma}^{\bot},\gamma(p_{\gamma})-\gamma \rangle	&= \cos(\vartheta)-1\,,\\
      \langle \dot{\gamma}(p_{\gamma})^{\bot},\gamma(m_{\gamma})-\gamma \rangle	&= \cos^2(\vartheta)-\cos(\vartheta)-\sin^2(\vartheta)\,,\\
      \langle \dot{\gamma}^{\bot},\gamma(m_{\gamma})-\gamma \rangle	&= \cos(\vartheta)-1\,,\\
      \langle \dot{\gamma}(m_{\gamma}),\gamma(p_{\gamma})-\gamma \rangle		&= -\sin(\vartheta)\big(1-2\cos(\vartheta)\big)\,,\\
      \langle \dot{\gamma}(m_{\gamma})^{\bot},\gamma(m_{\gamma})-\gamma \rangle	&= 1-\cos(\vartheta)\,,\\
      \langle \dot{\gamma}(m_{\gamma}),\gamma(m_{\gamma})-\gamma \rangle		&= -\sin(\vartheta)\,,\\
      \langle \dot{\gamma},\gamma(m_{\gamma})-\gamma \rangle			&= -\sin(\vartheta)\,,\\
      \langle \dot{\gamma}(m_{\gamma})^{\bot},\gamma(p_{\gamma})-\gamma \rangle	&= \cos^2(\vartheta)-\cos(\vartheta)-\sin^2(\vartheta)\,,\\
      \langle \dot{\gamma},\gamma(p_{\gamma})-\gamma(m_{\gamma}) \rangle		&= 2\sin(\vartheta)\;.
    \end{aligned}
  \end{equation*}
  In the following we treat the derivative in direction $a \dot{\gamma}^{\bot}$ and $b \dot{\gamma}$ separately.
  For the derivative in normal direction $a \dot{\gamma}^{\bot}$ we get
  \begin{align*}
    &2I'_{r}[\gamma](a\dot{\gamma}^{\bot})
    =
    2 \int_{m_{\gamma}}^{p_{\gamma}} a(\psi) d\psi - a \langle \dot{\gamma}^{\bot},\gamma(p_{\gamma})^{\bot}-\gamma(m_{\gamma})^{\bot} \rangle\\
    & \qquad
    - a(p_{\gamma}) \langle \gamma(p_{\gamma})-\gamma,\dot{\gamma}(p_{\gamma}) \rangle
    + a(m_{\gamma}) \langle \gamma(m_{\gamma})-\gamma,\dot{\gamma}(m_{\gamma})\rangle  \\
    & \qquad
    + \langle \gamma(p_{\gamma})-\gamma,\dot{\gamma}(p_{\gamma})^{\bot} \rangle
    \frac{\langle a\dot{\gamma}^{\bot}-a(p_{\gamma})\dot{\gamma}(p_{\gamma})^{\bot} ,\gamma(p_{\gamma})-\gamma \rangle}
         {\langle \dot{\gamma}(p_{\gamma}),\gamma(p_{\gamma})-\gamma \rangle}\\
         & \qquad
         - \langle \gamma(m_{\gamma})-\gamma,\dot{\gamma}(m_{\gamma})^{\bot} \rangle
         \frac{\langle a\dot{\gamma}^{\bot}-a(m_{\gamma})\dot{\gamma}(m_{\gamma})^{\bot},\gamma(m_{\gamma})-\gamma \rangle}
              {\langle \dot{\gamma}(m_{\gamma}),\gamma(m_{\gamma})-\gamma \rangle}\\
              & \qquad
              - \frac{r^2}{\sqrt{r^4 - \langle \gamma(p_{\gamma})-\gamma,\gamma(m_{\gamma})-\gamma \rangle^2}}\cdot\\
              & \qquad\quad \cdot
              \bigg( \!
              \frac{\langle a\dot{\gamma}^{\bot}-a(p_{\gamma})\dot{\gamma}(p_{\gamma})^{\bot},\gamma(p_{\gamma})-\gamma \rangle}
                   {\langle \dot{\gamma}(p_{\gamma}),\gamma(p_{\gamma})-\gamma \rangle}
                   \langle \dot{\gamma}(p_{\gamma}),\gamma(m_{\gamma})-\gamma \rangle \\
                   & \qquad \qquad
                   + \langle a(p_{\gamma})\dot{\gamma}(p_{\gamma})^{\bot}-a\dot{\gamma}^{\bot},\gamma(m_{\gamma})-\gamma \rangle
                   + \langle \gamma(p_{\gamma})-\gamma,a(m_{\gamma})\dot{\gamma}(m_{\gamma})^{\bot}-a\dot{\gamma}^{\bot} \rangle\\
                   & \qquad \qquad
                   + \frac{\langle a\dot{\gamma}^{\bot}-a(m_{\gamma})\dot{\gamma}(m_{\gamma})^{\bot},\gamma(m_{\gamma})-\gamma \rangle}
                   {\langle \dot{\gamma}(m_{\gamma}),\gamma(m_{\gamma})-\gamma \rangle}
                   \langle \gamma(p_{\gamma})-\gamma,\dot{\gamma}(m_{\gamma})\rangle
                   \! \bigg)\;.\\
  \end{align*}
  Inserting the expressions for $\gamma$ equal to the unit circle that have been calculated previously, we obtain
  \begin{align*}
    &2I'_{r}[\gamma](a\dot{\gamma}^{\bot})\\
    &= 2 \int_{\varphi-\vartheta}^{\varphi+\vartheta} a(\psi) d\psi \\
    & \qquad
    - \sin(\vartheta) \big( a(p_{\gamma}) + 2a + a(m_{\gamma}) \big)
    - \frac{(\cos(\vartheta)-1)^2 }{\sin(\vartheta)} \big( a(p_{\gamma}) + 2a + a(m_{\gamma}) \big)\\
    & \qquad
    -\frac{1}{\sin(\vartheta)}
    \Big( \!
    \big( a + a(p_{\gamma}) \big) \big( \! \cos(\vartheta)-1 \big) \big( 1-2\cos(\vartheta) \big)-2a(\cos(\vartheta)-1)\\
    & \qquad \qquad
    + \big( a(p_{\gamma}) + a(m_{\gamma}) \big) \big( \! \cos^2(\vartheta)-\cos(\vartheta)-\sin^2(\vartheta) \big)\\
    & \qquad \qquad
    + \big( a + a(m_{\gamma}) \big) \big( \! \cos(\vartheta)-1 \big) \big( 1-2\cos(\vartheta) \big)
    \! \Big)\\
    &= 2 \int_{\varphi-\vartheta}^{\varphi+\vartheta} a(\psi) d\psi \\
    & \qquad
    + \frac{\cos(\vartheta)-1}{\sin(\vartheta)} \big( 2a(p_{\gamma}) + 4a + 2a(m_{\gamma}) - 2a(p_{\gamma}) + 4\cos(\vartheta)a - 2a(m_{\gamma}) \big)\\
    &= 2 \int_{\varphi-\vartheta}^{\varphi+\vartheta} a(\psi) d\psi
    + 4a \frac{(\cos(\vartheta)-1)(\cos(\vartheta)+1)}{\sin(\vartheta)}\\
    &= 2 \int_{\varphi-\vartheta}^{\varphi+\vartheta} a(\psi) d\psi - 4a\sin(\vartheta) \; .
  \end{align*}
  Inserting
  the formulas above in the expression for $I'_{r}[\gamma](b \dot{\gamma})$ yields
  \begin{align*}
    &2 I'_{r}[\gamma](b\dot{\gamma})
    =
    2 \int_{m_{\gamma}}^{p_{\gamma}} b(\psi) \langle \dot{\gamma}(\psi),\dot{\gamma}(\psi)^{\bot} \rangle d\psi \\
    & \quad
    - \langle b \dot{\gamma},\gamma(p_{\gamma})^{\bot}-\gamma(m_{\gamma})^{\bot} \rangle
    + \langle \gamma(p_{\gamma})-\gamma,b(p_{\gamma})\dot{\gamma}(p_{\gamma})^{\bot} \rangle\\
    & \quad
    - \langle \gamma(m_{\gamma})-\gamma,b(m_{\gamma})\dot{\gamma}(m_{\gamma})^{\bot} \rangle
    + \langle \gamma(p_{\gamma})-\gamma,\dot{\gamma}(p_{\gamma})^{\bot} \rangle
    \frac{\langle b\dot{\gamma}-b(p_{\gamma})\dot{\gamma}(p_{\gamma}),\gamma(p_{\gamma})-\gamma \rangle}
	 {\langle \dot{\gamma}(p_{\gamma}) ,\gamma(p_{\gamma})-\gamma \rangle}\\
    & \quad
	 - \langle \gamma(m_{\gamma})-\gamma,\dot{\gamma}(m_{\gamma})^{\bot} \rangle
	 \frac{\langle b\dot{\gamma}-b(m_{\gamma})\dot{\gamma}(m_{\gamma}),\gamma(m_{\gamma})-\gamma \rangle}
	      {\langle \dot{\gamma}(m_{\gamma}) ,\gamma(m_{\gamma})-\gamma \rangle}\\
	      & \quad
	      - \frac{r^2}{\sqrt{r^4 - \langle \gamma(p_{\gamma})-\gamma,\gamma(m_{\gamma})-\gamma \rangle^2}}\cdot\\
              & \qquad\quad
	      \cdot\bigg( \!
	      \frac{\langle b\dot{\gamma}-b(p_{\gamma})\dot{\gamma}(p_{\gamma}),\gamma(p_{\gamma})-\gamma \rangle}
		   {\langle \dot{\gamma}(p_{\gamma}),\gamma(p_{\gamma})-\gamma \rangle}
		   \langle \dot{\gamma}(p_{\gamma}),\gamma(m_{\gamma})-\gamma \rangle \\
		   & \qquad \qquad
		   + \langle b(p_{\gamma})\dot{\gamma}(p_{\gamma})-b\dot{\gamma},\gamma(m_{\gamma})-\gamma \rangle
		   + \langle \gamma(p_{\gamma})-\gamma,b(m_{\gamma})\dot{\gamma}(m_{\gamma})-b\dot{\gamma} \rangle\\
		   & \qquad \qquad
		   + \frac{\langle b\dot{\gamma}-b(m_{\gamma})\dot{\gamma}(m_{\gamma}),\gamma(m_{\gamma})-\gamma \rangle}
		   {\langle \dot{\gamma}(m_{\gamma}),\gamma(m_{\gamma})-\gamma \rangle}
		   \langle \gamma(p_{\gamma})-\gamma,\dot{\gamma}(m_{\gamma}) \rangle
		   \! \bigg)\\
		   &= 0 - 0 + \big( 1-\cos(\vartheta) \big) \big( b(p_{\gamma})-b(m_{\gamma})+b-b(p_{\gamma})-b+b(m_{\gamma}) \big)\\
		   & \quad -\Big( \! \big( 1-2\cos(\vartheta)\big) \big( b-b(p_{\gamma})+b(p_{\gamma})-b(m_{\gamma})+b(m_{\gamma})-b \big) +(b-b) \! \Big)\\
		   &= 0 \; .
  \end{align*}
  Therefore,
  \begin{equation*}
    I'_{r}[\gamma](a\dot{\gamma}^{\bot} + b\dot{\gamma}) = \chi_{[-\vartheta,\vartheta]} \ast a - 2\sin(\vartheta) a\;.
    \qedhere
  \end{equation*}
\end{proof}

\bibliographystyle{plain}

\end{document}